\documentclass[12pt, a4paper]{article}

\usepackage[latin1]{inputenc}

\usepackage{amsmath, amssymb, amscd, amsfonts}

\usepackage{theorem}

\newtheorem{teo}{Theorem}[section]

\newtheorem{lemma}[teo]{Lemma}

\theorembodyfont{\normalfont}

\newtheorem{defi}[teo]{Definition}

\newtheorem{rmk}[teo]{Remark}

\newtheorem{es}[teo]{Example}

\newenvironment{proof}{\textit{Proof:}}{\begin{flushright} $\blacksquare$ \end{flushright}}

\newcommand {\fs}{\mathcal{O}}

\newcommand {\p}{\mathbb{P}^1}

\newcommand {\fr}{\rightarrow}

\renewcommand {\cal}{\mathcal}

\newcommand {\sym}{\text{Sym}}

\begin{document}

\title{Fibred surfaces with general pencils of genus 5}
\author{Elisa Tenni}
\date{}
\maketitle
\begin{abstract}\noindent
Let $f\!:S \fr B$ be a surface fibration with fibres of genus 5. We find a linear relation between the fundamental invariants of the surface. Namely $K_f^2=\chi_f+N$ where $N$ is the number of trigonal fibres. Our proof is based on the analysis of the relative canonical algebra $\cal{R}(f)$.
\end{abstract}

\section{Introduction}

Let $f\!: S \fr B$ be a morphism between a smooth projective surface and a smooth projective curve over an algebraically closed field $k$ of characteristic 0. This morphism is always flat (see \cite{Hart}, Proposition III.9.7).
For $P \in B$ write $F_P=f^{*}(P)$ for a fibre of $f$. Let $g=\text{genus}(F) \geq 2$ and $b=$ genus($B$).

Since char$(k)=0$ Ramanujan's Lemma (see \cite{BPV}, Proposition III
11.1) implies that $h^0(\cal{O}_F$) is constant. In particular we can suppose that the surface $S$ is connected and that $h^0(\cal{O}_F)=1$.

A morphism with these properties is called a \textit{surface fibration}. A fibration which has no $-1$-curves in any of its fibres is called \textit{relatively
minimal}. We will always assume that this is the case.\\

Let $\omega_f$ be the relative canonical bundle given by $\omega_S \otimes f^*\omega_B^{-1}$ and $K_f$ the associated divisor. By Arakelov's
theorem (see \cite{Bea}) $\omega_{S|B}$ is nef, that is, the intersection number
$K_f\cdot C$ is non negative for every irreducible curve $C$ on $S$. Furthermore
it is known that $K_f\cdot C= 0$ if and only if $C$ is a $-2$-curve contained
in a fiber.

The fundamental invariants associated to $S$ are

\begin{enumerate}
\item the self-intersection $K_f^2$ of the relative canonical divisor,
\item the relative Euler characteristic $\chi_f= \chi (\cal{O}_S)- (b-1)(g-1)$,
\item the relative topological Euler characteristic \[e_f=e(S) - e(F)\,e(B)=\sum_{b \in B}( e(F_b)-2+2g)\] where $F$ is a general fiber of $f$ and
$e(X)$ denotes the topological Euler number of the space $X$.
\end{enumerate}

The three invariants are related by Noether's Formula:
\begin{equation}
K_f^2+e_f=12 \chi_f,
\end{equation}
thus only two of them are independent. \\

In this paper we study the geography of a fibred surface with fibres of genus 5. Our main result Theorem \ref{enunciato} gives a linear relation between the fundamental invariants of the surface and the number of the trigonal fibres.\\

The techniques we use are based on the analysis of the \textit{relative canonical algebra} of the fibration $f$

\begin{equation}
\cal{R}(f)=\bigoplus_{n \geq 0} \cal{R}_n
\end{equation}
where
\begin{equation}
\cal{R}_n=f_* \omega_{S|B}^{\otimes n}.
\end{equation}

Similar ideas have been applied successfully to the study of surface fibrations of genus 2 and 3. In her PhD thesis Mendes Lopes completed the local analysis of the canonical algebra of curves of genus 2 and 3 (see \cite{ML}). Based on this, works of Reid (\cite{Reid1}), Ashikaga and Konno (\cite{AK}) and Catanese and Pignatelli (\cite{CP}) led to a proof of the following statement:

\begin{teo}[Horikawa, Gang, Reid]
Let $f\!: S \fr B$ a surface fibration with fibres of genus 2. Then
\begin{equation}
K_f^2
= 2\chi_f  +\sum_{P \in B}
H(P)
\end{equation}
where the Horikawa number of a genus 2 fibre germ over $P$
\[H(P)=H(P)=\operatorname{lenght} (\operatorname{coker}(\text{\emph{Sym}}^2\cal{R}_1 \fr \cal{R}_2)\emph{)}\]
can be interpretated (roughly speaking) as the virtual number of  $\,2$-disconnected fibres of type $E_1 + E_2$'' (with
$E1$, $E2$ elliptic curves meeting transversally in one point).\\

Let $f\!: S \fr B$ a surface fibration with fibres of genus 3. Then
\begin{equation}
K_f^2
= 3\chi_f  +\sum_{P \in B}
H(P)
\end{equation}
where the Horikawa number $H(P)$ is defined as \[H(P)=\operatorname{lenght} (\operatorname{coker}(\text{\emph{Sym}}^2\cal{R}_1 \fr \cal{R}_2)\emph{)}.\]
\end{teo}

\begin{defi}[\cite{AK}]
Let $f\!: S \fr B$ a surface fibration with a certain condition on its general fibre. If there exists a rational number $\lambda$, a finite set of fibres $F_1, \ldots, F_n$ and well-defined nonnegative rational numbers $\operatorname{Ind}(F)$ satisfying
\[K_f^2= \lambda \chi_f + \sum_{i=1}^n \operatorname{Ind}(F_i)\]
we call the relation a \textit{slope equality}.\\
The fibres with positive index Ind are called the \textit{atoms} of the fibration. 
\end{defi}

\begin{rmk}
The essence of the slope equality is the concentration of the global invariants of the surface on these fibres. In the case of genus 2 and 3 fibrations the index $\operatorname{Ind}(F_P)$ is the Horikawa number $H(P)$, thus the atoms are the fibres with positive Horikawa number.
\end{rmk}

\begin{defi}
An effective divisor $D$ on a smooth algebraic surface
$S$ is said to be \textit{$k$-connected} if, whenever we write $D = A+B$ as a sum
of effective divisors $A, \, B $, we have that $A.B \geq k$.
\end{defi}
\begin{defi}
An effective $2$-connected divisor on an algebraic surface is \textit{trigonal} if it has a $g_3^1$.
\end{defi}
\begin{rmk}
The locus of trigonal curves has codimension 1 in the moduli space of genus 5 stable curves $\overline{\cal{M}}_5$, while the general curve has gonality 4.
\end{rmk}

\begin{teo}\label{enunciato}
Let $S$ be a projective smooth surface and $B$ a projective smooth curve. Let $f\!: S \fr B$ be a
surface fibration of genus $5$ and let us suppose that every fibre is $3$-connected and nonhyperelliptic. Let us also suppose that the general fibre is nontrigonal.

Then 

\begin{equation}\label{teo}
K_f^2=4 \chi_f + N
\end{equation}
where $N$ is the number of the trigonal fibres.

\end{teo}

\begin{rmk}This theorem proves a slope equality for genus 5 fibrations with maximal gonality: there is a linear relation between $K_f^2$ and $\chi_f$ with a correction term which takes into account the existence of the atoms which, in our case, are the trigonal fibres.
\end{rmk}
\begin{rmk}In particular, the fibrations we are dealing with will be always relatively minimal, since the hypothesis that the fibres are 3 connected implies that $K_S$ is relatively ample (see Lemma \ref{lemma}); in particular the fibres of $f$ do not contain $-1$-curves or $-2$-curves.\end{rmk}
\begin{rmk} We require that every fibre is 3-connected and nonhyperelliptic. This hypothesis is important for the techniques we use, because it implies the existence of the canonical embedding of every fibre. However most singular stable curves are not even 2-connected. On the other hand our fibres can have arbitrary singularities.

It would be interesting to understand what happens if we drop this assumptions, at least for stable fibrations. We expect that the study of the trigonal divisor in $\overline{\cal{M}}_5$ will lead to a generalization of Theorem \ref{enunciato}.
\end{rmk}

Under the assumptions of our theorem $S$ is a relative canonical model (it is smooth and $K_S$ is relatively ample), so \cite{Reid2}
\[S= \text{Proj}_B (\cal{R}(f))\] and \[S \subset \mathbb{P}= \text{Proj}_B (\bigoplus_{n \geq 0} \text{Sym}^n f_* \omega_{S|B}).\]

$\mathbb{P}$ is a $\mathbb{P}^4$-bundle over $B$ and fibrewise the embedding $S \subseteq \mathbb{P}$ restricts to the canonical embedding of the curve. $S$ has codimension 3 in $\mathbb{P}$ and so the results in \cite{BE} imply that, at least locally, the equations are in the Pfaffian form. In Section \ref{Examples} we use this fact to produce examples of regular surfaces satisfying the assumptions of Theorem \ref{enunciato}, for every possible value of $p_g \geq 0$.\\

\textbf{Acknowledgements.} I would like to thank my supervisor Miles Reid for introducing me to this beautiful topic and for all the help he generously gave me. I am grateful to the geometry groups at Warwick University and Pavia for their help and support. Thanks to R. Pignatelli and L. Stoppino for many useful suggestions.

\section{The relative canonical algebra of a fibration}
Let $K_f$ be the
relative canonical divisor $K_S-f^*K_B$, $\omega_{S|B}$ the corresponding sheaf and set
\[\cal{R}_n=f_*(\omega_{S|B}^{\otimes n}) \qquad \text{for} \; n \geq 0.\]

Since $F^2=0$ for any fibre $F$ we see that \[K_F=(K_S+F)_{|F}=(K_f)_{|F}.\]
By Serre duality $h^0(F, \cal{O}_F)=h^1(F, \omega_F)=1$, while $\chi(\omega_F)$ is constant by flatness. Thus $h^0(F, \omega_F)$ is constant and we call this common value $g$. In addition $h^0(F, \omega_F^{\otimes n})=(2n-1)(g-1)$ for $n>1$.

The fibres of $\cal{R}_n$
\[\cal{R}_n \otimes k(b)=H^0(F_b,{(\omega_{S|B}^{\otimes n}})_{|F_b} )=H^0(F, \omega_F^{\otimes n})\]
have constant dimension, so by base change (see \cite{Mum}, Corollary II.2), $\cal{R}_n$ is a locally free sheaf. The rank of $\cal{R}_n$ is given by

\begin{equation}\label{rank}\text{rank} \cal{R}_n=\left\{ \begin{array}{ll}
1 & n=0\\
g & n=1\\
(2n-1)(g-1) & n >1\\
\end{array}\right.\end{equation}

\begin{defi}The \textit{relative canonical algebra} is defined as
\[\cal{R}(S/B)=\cal{R}(f)=\bigoplus_{n \geq 0} \cal{R}_n\]
with multiplication induced by the tensor product $\omega_{S|B}^{\otimes n} \otimes \omega_{S|B}^{\otimes m}
\fr \omega_{S|B}^{\otimes n+m}$. It is a finitely generated $\cal{O}_B$-algebra, generated in degree
$\leq 4$ (the 1-2-3 Theorem, see \cite{K}).\end{defi}

We can easily calculate $\chi(\cal{R}_n)$. By the Leray spectral sequence \[H^n(S, \cal{F})=\bigoplus_{p+q=n}H^p(B, R^qf_*\cal{F})\] for any coherent sheaf $\cal{F}$ on
$S$. Again by base change we find that \[R^1f_*\omega_{S|B}^{\otimes n}=0 \text{ for all }n>1\] and
\[R^2f_*\omega_{S|B}^{\otimes n}=0 \text{ for all }n \geq 1.\] Moreover Grothendieck duality implies that \[R^1f_* \omega_{S|B}=\fs_B.\]
Putting these facts together we find

 \begin{equation} \left\{ \begin{array}{ll}

\chi(\cal{R}_1)=\chi(\omega_{S|B})-\chi(\fs_B)&\\
\chi(\cal{R}_n)=\chi(\omega_{S|B}^{\otimes n})& \text{for all  } \, n >1.

\end{array}\right.\end{equation}

Using the Riemann-Roch formula for surfaces and curves, and defining
\[\chi_f=\chi(\cal{O}_S)-(b-1)(g-1)\] we finally compute

\begin{eqnarray} \label{chi}
\chi(\cal{R}_n) &=&\left\{ \begin{array}{ll}
b-1 & n=0,\\
\chi_f+g(1-b)& n=1,\\
\chi_f + {n \choose 2} K_f^2+(2n-1)(g-1)(1-b)& n >1;\\
\end{array}\right.\\[6pt]\label{deg}
 \deg(\cal{R}_n) &=&\left\{ \begin{array}{ll}
0 & n=0,\\
\chi_f& n=1,\\
\chi_f + {n \choose 2} K_f^2& n >1.\\
\end{array}\right. \end{eqnarray}

\section{Fibrations of genus 5}
In this section we present the proof of Theorem \ref{enunciato}. The following Lemma will clarify the nature of our assumptions.

\begin{lemma}\label{lemma}
Let $C$ be an effective divisor on a smooth surface. Let us suppose that $C$ is $3$-connected and nonhyperelliptic.\\

Then the dualizing sheaf $\omega_C=(\omega_S\otimes \cal{O}_S(C))_{|C}$ is very ample. Moreover Noether's Theorem holds, i.e.
\[\emph{Sym}^n H^0(C, \omega_C) \fr H^0(C, \omega_C^{\otimes n})\]
is surjective for any $n \geq 1$.\\

Also Petri's Theorem holds. In the case of a genus 5 divisor it says that the ideal of its canonical embedding in $\mathbb{P}^4$ is defined by three quadrics if the curve is nontrigonal, and by three quadrics and two cubics if it is trigonal.
\end{lemma}

\begin{proof}
It is proved in \cite{CFHR} that, given an effective divisor $C$ on a smooth algebraic surface, the dualizing sheaf $\omega_C$ is very ample if and only if the curve is $3$-connected and not honestly hyperelliptic (which means that there is a double cover of $\mathbb{P}^1$ induced by the canonical morphism).

Even if the fibre is singular, the existence of the canonical
embedding is enough to prove that Noether's Theorem and Petri's Theorem are still valid, for instance the proof in \cite{ACGH} go through without change. \end{proof}

Consider the map \[\phi_n\!: \text{Sym}^n(\cal{R}_1) \fr \cal{R}_n\] Looking
at the stalk in a point $b \in B$, the map $\phi_{n}$ is just

\[\sym^n H^0(F, \omega_F) \fr H^0(F, \omega_F^n).\]
By Noether's Theorem this map of stalks is surjective, thus $\phi_n$ is surjective.

We are interested the locally free sheaves $\cal{K}_n= \text{ker}( \phi_n)$. For each $b \in
B$, the fibre $\cal{K}_n \otimes k(b)$ is the vector space of polynomials of degree $n$ in $\mathbb{P}^4$
vanishing on the curve $F_b$:
\begin{equation}
\cal{K}_n \otimes k(b)= H^0(\mathbb{P}^4, I_{F_b}(n)).
\end{equation}
In particular one can see that rank$(\cal{K}_2)=3$ and rank$(\cal{K}_3)=15$.\\

Let us consider the commutative diagram:
\begin{equation}\label{diagramma}\begin{matrix}

0 & \fr & \cal{K}_2\otimes \cal{R}_1 & \fr & \text{Sym}^2 \cal{R}_1 \otimes \cal{R}_1 &
\fr & \cal{R}_2 \otimes \cal{R}_1 & \fr & 0\\
& & \downarrow & & \downarrow & & \downarrow \\
0 & \fr & \cal{K}_3 & \fr & \text{Sym}^3 \cal{R}_1 &
\fr & \cal{R}_3 & \fr & 0\\

\end{matrix}\end{equation}

The map $\mu\!: \cal{K}_2 \otimes \cal{R}_1 \fr \cal{K}_3  $ induces a map of fibres over $b \in B$
\[\mu_b\!: H^0(\mathbb{P}^4, I_{F_b} (2))\otimes H^0(\mathbb{P}^4,\cal{O}_{\mathbb{P}^4}(1))  \fr H^0(\mathbb{P}^4, I_{F_b} (3)).  \]
We can apply Lemma \ref{lemma} and see that for any point $b$ supporting a nontrigonal fibre, $\mu_b$ is surjective, since the ideal $I_{F_b}$ is generated by quadrics. In particular $\mu_b$ is an isomorphism since the two vector spaces both have dimension 15 over the base field $k$. Conversely, over a point supporting a trigonal fibre, the cokernel of $\mu_b$ is a 2 dimensional vector space.

The kernel of the map of sheaves $\mu$ is trivial, while the cokernel is a skyscraper sheaf $\cal{F}$:

\begin{equation}\label{quad e cubiche}
0 \fr \cal{K}_2 \otimes \cal{R}_1 \fr \cal{K}_3 \fr \cal{F} \fr 0.
\end{equation}

The stalk of $\cal{F}$ is a two dimensional vector space on the points supporting a trigonal fibre, while it is trivial elsewhere. It is then clear that \begin{equation}\label{2N}\chi(\cal{F})=2N\end{equation} where $N$ is the number of the trigonal fibres.

\vspace{0,5cm}

\textit{Proof of Theorem \ref{enunciato}}\\

Taking into account the formulae (\ref{chi}) and (\ref{deg}) it is possible to calculate the Euler characteristic of $\text{Sym}^2(\cal{R}_1)\otimes \cal{R}_1$ and $\text{Sym}^3(\cal{R}_1)$ using the splitting principle:
\begin{eqnarray}
\deg (\text{Sym}^2(\cal{R}_1))&=& 6 \chi_f, \\
\deg (\text{Sym}^3(\cal{R}_1))&=& 21 \chi_f. \nonumber
\end{eqnarray}

Using standard results about the Chern classes of a tensor product, one can calculate that
\begin{eqnarray}
\deg (\text{Sym}^2(\cal{R}_1)\otimes \cal{R}_1)&=& 45 \chi_f \\
\deg (\cal{R}_2\otimes \cal{R}_1)&=& 17 \chi_f+5 K_f^2 \nonumber
\end{eqnarray}
and if we take into account diagram (\ref{diagramma}) we can conclude that
\begin{eqnarray}\label{chi strani}
\chi(\cal{K}_3)&=&20 \chi_f-3 K_f^2 +15 \chi(\cal{O}_B)\\
\chi(\cal{K}_2\otimes \cal{R}_1)&=&28 \chi_f-5 K_f^2 +15 \chi(\cal{O}_B).\nonumber
\end{eqnarray}

Consider now the exact sequence (\ref{quad e cubiche}). By additivity of the Euler characteristic and substituting the equalities (\ref{2N}) and (\ref{chi strani}) we find
\begin{eqnarray*} 0 &= &\chi (\cal{R}_2 \otimes \cal{R}_1)- \chi (\cal{R}_3)+ \chi (\cal{F})\\
&=& (28 \chi_f-5 K_f^2 +15 \chi(\cal{O}_B))-(20 \chi_f-3 K_f^2 +15 \chi(\cal{O}_B))+2N\\
&=& 8 \chi_f-2 K_f^2 + 2N.
\end{eqnarray*}
This is precisely the statement of our Theorem.
\begin{flushright} $\blacksquare$ \end{flushright}

\section{Examples}\label{Examples}

In this Section we exhibit examples of fibrations satisfying the assumptions of Theorem \ref{enunciato}.
In particular we are looking for fibrations $f\!: S \fr \p$ with $S$ smooth projective surface and a single trigonal fibre.

\begin{es}\label{ex} The easiest ambient space $\mathbb{P}$ we can work with is $\p_{t_0:t_1}\times \mathbb{P}^4_{x_0:\,\ldots\, :x_4}$. $S$ has codimension 3 in $
 \p \times \mathbb{P}^4$ and, bearing \cite{BE} in mind, we expect the equations of $S$ to have Pfaffian form.\end{es}

By Petri's Theorem, the generic fibre is a complete intersection of three quadratic polynomials in $\mathbb{P}^4$, while for the trigonal fibres we need 2 more cubic polynomials. These must be in the ideal generated by the quadric
polynomials for every other fibre.

Consider the Pfaffian equations of the following skew
symmetric matrix (we are writing only the upper triangular part, since al the other entries are
determined by skew-symmetry):

\begin{equation}
M=\left( \begin{matrix}
  & t_1 & x_0 & x_2 & x_3\\
 & & t_0x_1 & t_0x_3 & t_0x_4\\
 & & & q_1 & q_2\\
 & & & & q_3
\end{matrix} \right)
\end{equation}
 $q_1,\,q_2,\,q_3$ are generic quadratic homogeneous polynomials only in the $x$ variables. For those unfamiliar with the Pfaffian notation, the Pfaffian equations of the matrix $M$ are:

\begin{equation} \label{pfaff}\left \{\begin{array}{lll}

c_1&=& t_0(x_1q_3-x_3q_2+x_4q_1)\\
c_2&=&x_0q_3-x_2q_2+x_3q_1\\
p_1&=&t_1q_3+t_0( - x_2x_4+x_3^2)\\
p_2&=&t_1q_2+t_0(-x_0x_4+x_1x_3)\\
p_3&=&t_1q_1+t_0(-x_0x_3+x_1x_2)

\end{array}\right.\end{equation}

Over each $(t_0, t_1) \neq (1,0)$ the two cubic polynomials $c_1$ and $c_2$ in (\ref{pfaff}) are linear combinations of
the three quadric polynomials:
\begin{eqnarray*}
t_1 c_1&=& t_0(x_1 p_1 - x_2 p_2+x_4 p_3)\\
t_1 c_2&=& x_0 p_1 - x_2 p_2+x_3 p_3.
\end{eqnarray*}

Over $(t_0, t_1) = (1,0)$ we impose a trigonal fibre. But we know
that a nonsingular trigonal curve of genus 5 in $\mathbb{P}^4$ is the intersection of three cubic polynomials
and a rational normal scroll $\mathbb{F}(1,2)$ (see \cite{Reid1}). The equations of $\mathbb{F}(1,2)\subseteq \mathbb{P}^4$ are
\begin{equation*}\text{rank} \left( \begin{matrix}
 x_0 & x_2 & x_3\\
 x_1 & x_3 & x_4\\
\end{matrix} \right)\leq 1\end{equation*}
These equations coincide with $p_1, \, p_2,\, p_3$ when $t_1=0$. One can check that by \cite{Reid1} the equations $c_1$ and $c_2$ cut a trigonal curve inside the scroll $\mathbb{F}(1,2)$.\\

The only issue now is to choose $q_1, q_2, q_3$ such that $S$ is a nonsingular surface. In order to do this we apply
Bertini's Theorem over the subset $(t_1 \neq 0) \subset \p \times \mathbb{P}^4$ and conclude that
the three quadric polynomials are general enough to have nonsingular intersection. Thus $S$ is smooth away from the trigonal fibre. Then we apply Bertini's Theorem again over $t_1=0$ to conclude that the trigonal fibre is smooth for a generic triple $q_1,\, q_2,\, q_3$. Thus for the generic triple $q_1,\, q_2,\, q_3$ $S$ is a nonsingular surface.\\

In this situation the projection $f\!: S \fr \p$ is a flat morphism, because it is a surjective
morphism between a smooth surface and a curve. Moreover, every fibre is connected since $h^0(F,
\cal{O}_F)$ is constant and we have chosen the fibre over $t_1=0$ to be a smooth connected curve. So every
fibre is a canonical genus 5 connected curve, thus we have no hyperelliptic fibres. They must be all
3 connected. In fact being 3 connected and nonhyperelliptic is equivalent to the existence of the canonical
embedding, and all the fibres are canonically embedded (see \cite{CFHR}).\\

Let us compute the invariants of the surface $S$. For this we need the free resolution of $S$ inside $\p \times \mathbb{P}^4$.

One can prove this is given by
\begin{equation}\label{risoluzione}
\begin{matrix}
0 \leftarrow \cal{O}_S \leftarrow \cal{O}_{\p \times \mathbb{P}^4} \\ \stackrel{\text{Pf}
\,M}{\longleftarrow} \cal{O}_{\p \times \mathbb{P}^4}(-1,-3)\oplus \cal{O}_{\p \times
\mathbb{P}^4}(0,-3) \oplus 3\cal{O}_{\p \times \mathbb{P}^4}(-1,-2)\\
 \stackrel{M} {\leftarrow} \cal{O}_{\p \times \mathbb{P}^4}(-1,-3)\oplus \cal{O}_{\p \times
\mathbb{P}^4}(-2,-3) \oplus 3\cal{O}_{\p \times \mathbb{P}^4}(-1,-4) \\ \stackrel{\text{Pf}
\,M}{\longleftarrow}\cal{O}_{\p \times \mathbb{P}^4}(-2,-6)\leftarrow 0\end{matrix}\end{equation}

Exactness can be checked on the fibres, thus one only needs to work out the free resolution of the canonical image of a nonhyperelliptic genus 5 curve. But this is straightforward both in the nontrigonal and in the trigonal case.\\

The dualising sheaf $\omega_S$ can be computed by dualising the free resolution of $\cal{O}_S$ (see \cite{Hart}, Proposition III.7.5). We conclude that
$\omega_S=\cal{O}_S(0,1)$ and that there is an isomorphism  $H^0(\p \times \mathbb{P}^4,
\cal{O}(0,1))\stackrel{\sim}{\fr}H^0(S,\omega_S)$.

In particular we see that $p_g=5$.\\

We can prove as well that the surface is regular. To prove this one considers again the resolution (\ref{risoluzione}) and check that most of the cohomology groups involved vanish.

The upshot of this is that $\chi_f=1+p_g+4=10$.\\

It remains to compute $K_f^2$.

The relative canonical sheaf $\omega_{S|\p}=\omega_S\otimes f^*\cal{O}(2)=\cal{O}_S(2,1)$ is very
ample so for $n$ big enough, $H^i(S, \omega_{S|B}^{\otimes n})=0$ $\forall\, i >0$.
Once again Leray spectral sequences give $H^i(S, \omega_S^{\otimes n})=H^i(\p,
\cal{R}_n)$ $\forall\, i, \; \forall\, n>1$.

Thus for $n$ big enough we get $\chi(\cal{R}_n)= h^0(S, \cal{O}_S(2n,n))$. We can calculate
the latter tensoring the resolution (\ref{risoluzione}) with $\cal{O}(2n,n)$, obtaining a new exact
sequence, which yields
\[\chi(\cal{R}_n)=6-\frac{25n}{2}+\frac{41n^2}{2}\]
at least for big $n$.
We know that ${n\choose2} K_f^2=\chi({R}_n)-\chi_f-4(2n-1)=\frac{41}{2} n(n-1)$ (equation (\ref{chi})), and it
immediately follows that $K_f^2=41$. This is exactly the thesis of Theorem \ref{enunciato}. \\

\begin{rmk}The surface $S$ can be seen as a complete intersection in
$\mathbb{P}^4$:
\begin{equation*}
\begin{matrix}
S & \fr & \mathbb{P}^4\\
(t_0,t_1),(x_0, \ldots, x_4) & \mapsto & (x_0, \ldots, x_4)
\end{matrix}
\end{equation*}
Its image obviously lies in the surface \[\tilde{S}=\{x_0 q_3 - x_2 q_2+x_3q_1=0, x_1 q_3 - x_3
q_2+x_4q_1=0\}.\] The morphism $S \fr \tilde{S}$ is in fact an isomorphism, as it has an inverse. This
can be shown by calculation for $q_1, \,q_2, \, q_3$ general enough.\end{rmk}

\begin{es}
The previous example can be adapted to a more general ambient space, namely a normal rational scroll
$\mathbb{F}(a_0,a_1,a_2,a_3,a_4)$. In particular we can easily find examples for any odd $p_g \geq 5$.\end{es}

Let us fix $\mathbb{F}=\mathbb{F}(a,a,0,0,0)$ for any $a \geq 0$. We can consider the Pfaffian
equations of a matrix very similar to the one of the previous example:

\begin{equation}
M=\left( \begin{matrix}
  & t_1 & t_0^a x_0 & x_2 & x_3\\
 & & t_0^{a+1} x_1 & t_0x_3 & t_0x_4\\
 & & & q_1 & q_2\\
 & & & & q_3
\end{matrix} \right)
\end{equation}
with Pfaffian equations

\begin{equation} \left \{\begin{array}{l}

c_1=t_0(t_0^a x_1q_3-x_3q_2+x_4q_1)\\
c_2=t_0^a x_0q_3-x_2q_2+x_3q_1\\
p_1=t_1q_3+t_0( - x_2x_4+x_3^2)\\
p_2=t_1q_2+t_0^{a+1}(x_0x_4-x_1x_3)\\
p_3=t_1q_1+t_0^{a+1}(x_0x_3-x_1x_2)

\end{array}\right.\end{equation}
where the $q_i$ have bidegree $(a,2)$. Like in Example \ref{ex}, we can find $q_i$ general enough to yield a
smooth surface with one single trigonal fibre.

We can show again that the surface is regular and that $\omega_S=\cal{O}_S(0,1)$ and
$H^0(\mathbb{F}, \cal{O}_{\mathbb{F}}(0,1)) \stackrel{\sim}{\fr}H^0(S,\omega_S)$. So $p_g=2a+5$.

\begin{es}We can modify the latter example in order to obtain even $p_g\geq 6$. If $p_g$ is even and the surface is regular, the degree of $\cal{R}_1$ is odd, so we look for an ambient space $\text{Proj(Sym}^n \cal{R}_1)=\mathbb{F}(a_0,a_1,a_2,a_3,a_4)$ with $\sum_i a_i$ odd. \end{es}

Let us fix $\mathbb{F}=\mathbb{F}(a,0,0,0,0)$ for any $a$ positive odd integer of the form $2d-1$. The matrix involved is

\begin{equation}
M=\left( \begin{matrix}

  & t_1^{d+1} & t_0^{2d} x_0 & x_2 & x_3\\
 & & t_0^{d+1} x_1 & t_0^dx_3 & t_0^dx_4\\
 & & & q_1 & q_2\\
 & & & & q_3
\end{matrix} \right)
\end{equation}
with Pfaffian equations

\begin{equation} \left \{\begin{array}{l}

c_1=t_0^d(t_0 x_1q_3-x_3q_2+x_4q_1)\\
c_2=t_0^{2d} x_0q_3-x_2q_2+x_3q_1\\
p_1=t_1^{d+1}q_3+t_0^d( - x_2x_4+x_3^2)\\
p_2=t_1^{d+1}q_2+t_0^{d+1}(t_0^{2d-1}x_0x_4-x_1x_3)\\
p_3=t_1^{d+1}q_1+t_0^{d+1}(t_0^{2d-1}x_0x_3-x_1x_2)

\end{array}\right.\end{equation}
where $q_1$ and $q_2$ have bidegree $(0,2)$ and $q_3$ has bidegree $(-1,2)$.

We can show again that the surface is regular and that $\omega_S=\cal{O}_S(0,1)$ and $H^0(\mathbb{F}, \cal{O}_{\mathbb{F}}(0,1)) \stackrel{\sim}{\fr}H^0(S,\omega_S)$. So $p_g=a+5=2d+4$.

\begin{es} The only missing values for $p_g$ are $0,\ldots, 4$. These can be obtained by a suitable modification of the above construction. \end{es}

\vskip.3cm
\noindent Elisa Tenni\\
Dipartimento di Matematica ``F. Casorati''\\
Universit\`a degli Studi di Pavia\\
via Ferrata 1, 27100 Pavia, Italy\\
\textit{e-mail:} elisa.tenni\texttt{@unipv.it}

\end{document}